\documentclass[reqno]{amsart}
\usepackage{amsmath,amssymb,amsthm,mathrsfs,xcolor,dsfont}

\numberwithin{equation}{section}
\newtheorem{thm}{theorem}[section]

\newtheorem{Corollary}[thm]{Corollary}
\newtheorem{Lemma}[thm]{Lemma}
\newtheorem{Proposition}[thm]{Proposition}

\begin{document}
\title[NASC for global existence for periodic non-gauge invariant NLS]{Necessary and sufficient condition for global existence of $L^2$ solutions for 1D periodic NLS
with non-gauge invariant quadratic nonlinearity}
\author{Kazumasa Fujiwara}
\address{
K. Fujiwara
\newline
Graduate School of Mathematics, Nagoya University\\
Furocho, Chikusaku, Nagoya, 464-8602, Japan
}
\email{fujiwara.kazumasa@math.nagoya-u.ac.jp}
\author{Vladimir Georgiev}
\address{V. Georgiev
\newline Dipartimento di Matematica Universit\`a di Pisa
Largo B. Pontecorvo 5, 56127 Pisa, Italy\\
 and \\
 Faculty of Science and Engineering \\ Waseda University \\
 3-4-1, Okubo, Shinjuku-ku, Tokyo 169-8555 \\
Japan and IMI--BAS, Acad.
Georgi Bonchev Str., Block 8, 1113 Sofia, Bulgaria}%
\email{georgiev@dm.unipi.it}%

\begin{abstract}
We study 1D NLS with non-gauge invariant quadratic nonlinearity on the torus.
The Cauchy problem admits trivial global solutions
which are constant with respect to space.
The non-existence of global solutions also has been studied
only by focusing on the behavior of the Fourier $0$ mode of solutions.
However,
the earlier works are not sufficient
to obtain the precise criteria for the global existence for the Cauchy problem.
In this paper,
the exact criteria for the global existence of $L^2$ solutions
is shown by studying the interaction between the Fourier $0$ mode and oscillation of solutions.
Namely, $L^2$ solutions are shown a priori not to exist globally
if they are different from the trivial ones.
\end{abstract}

\maketitle

\section{introduction}

It is well - known that nonlinear dispersive equations may have different qualitative behaviour
as a result of the competition between nonlinear source terms and the dispersion of the solution.
If  $u(t,x)$ is the complex - valued solution of the nonlinear Schr\"odinger problem
$$ i \partial_t u + \Delta u = f(u),$$
where $t\geq 0$ and $x$ denote time and space variables, respectively,
then typical non linear source terms  are of type
\begin{alignat}{2}\label{eq.i1}
 &  f(u)= |u|^p, \\ & f(u) = \pm |u|^{p-1}u.
   \label{eq.i2}
\end{alignat}

The strength of the dispersion is measured by the quantity
\begin{equation}\label{eq.i1a}
   \sup_{t>1} t^\sigma \|u(t,x)\|_{L^\infty_x} < \infty, \sigma \geq 0 .
\end{equation}
If the space variables are in $\mathbb{R}^n,$ then the linear Schr\"odinger equation has dispersion parameter $\sigma = n/2 >0.$
On the other hand,
if the space variables are in a compact manifold, then $\sigma =0$ and there is a lack of dispersion of type \eqref{eq.i1a} with $\sigma >0.$

In this work,
we consider the simplest cases of $x \in \mathbb{T}$ and nonlinearity $f(u)=|u|^2$ and therefore we study  the Cauchy problem,
	\begin{align}
	\begin{cases}
	i \partial_t u + \Delta u = |u|^2,
	& t \in \lbrack 0, T),\ x \in \mathbb T,\\
	u(0,x) = \phi(x),
	& x \in \mathbb T
	\end{cases}
	\label{eq:1.1}
	\end{align}
for $T > 0$.
Our goal is to provide
a sharp condition of initial data
for the non-existence of global solutions.

The Cauchy problem
	\begin{align}
	\begin{cases}
	i \partial_t u + \Delta u = |u|^p,
	& t \in \lbrack 0, T),\ x \in \mathbb T^n,\\
	u(0,x) = \phi(x),
	& x \in \mathbb T^n
	\end{cases}
	\label{eq:1.2}
	\end{align}
with $n \geq 1$ and $p > 1$
is known to be locally well-posed
with sufficiently smooth initial data and suitable power of nonlinear term.
For example,
for $s > n/2$ and $p \in 2 \mathbb N$,
the local well-posedness of \eqref{eq:1.2} is shown in the $H^s$ framework.
Here the Sobolev space of order $s$ on the torus $\mathbb T^n$ is defined by
	\[
	H^s
	= \bigg\{ u \in L^2;\
	\sum_{k \in \mathbb Z^n} (1+|k|^2)^{s} |\hat u(k)|^2 < \infty \bigg\},
	\]
where $\displaystyle \hat u(k) = (2 \pi)^{-n} \int_{\mathbb T^n} e^{-ik \cdot x} u(x) dx$.
There are some literature on the detail of local well-posedness of \eqref{eq:1.2}.
We refer the reader to \cite{bib:1,bib:2,bib:3,bib:6} and references therein, for instance.

The non-existence of global solutions to \eqref{eq:1.2}
has been studied under some condition of $\hat \phi(0)$,
the Fourier $0$ mode of $\phi$ or integral of $\phi$.
Especially, the positivity of the nonlinearity has played an important role.
In \cite{bib:11},
for any $p > 1$,
Oh showed the non-existence of the global weak solutions to \eqref{eq:1.2}
if initial data $\phi$ satisfies
	\begin{align}
	\mathrm{Im} \int_{\mathbb T^n} \phi(x) dx < 0.
	\label{eq:1.3}
	\end{align}
Here we say $u \in L_{\mathrm{loc}}^p(0,\infty ; L^p(\mathbb T^n))$
is a global weak solution to \eqref{eq:1.2}
if for any $T>0$, $u$ satisfies the weak form
	\begin{align}
	&\int_0^T \int_{\mathbb T^n} u(t,x) ( - i \partial_t \psi(t,x) + \Delta \psi(t,x)) dx \ dt
	\nonumber\\
	&= i \int_{\mathbb T^n} \phi(x) \psi(0,x) dx
	+ \int_0^T \int_{\mathbb T^n} |u(t,x)|^p \psi(t,x) dx \ dt
	\label{eq:1.4}
	\end{align}
for any
$\psi \in C ([0,T] \times \mathbb T^n) \cap C^\infty((0,T) \times \mathbb T^n)$ with $\psi(T,x) = 0$ for any $x \in \mathbb T^n$.
The non-existence was shown
by a so-called test function method,
namely, it is shown that if $\phi$ satisfyies \eqref{eq:1.3},
then the weak form \eqref{eq:1.4} cannot be satisfied with some $T > 0$ and $\psi$.
We remark that the test function method
is a general method to show the non-existence of weak solutions
with the nonlinearity of the type $|u|^p$.
For the test function method and related topics,
we refer the reader \cite{bib:8,bib:13,bib:14} and references therein.
Moreover, in \cite{FO17},
the first author and Ozawa generalized the condition of Oh for smooth initial data.
\begin{Proposition}[{\cite[Proposition 1.3]{FO17}}]
\label{Proposition:1.1}
For any $n \geq 1$ and $p > 1$,
if initial data $\phi \in H^2(\mathbb T^n)$ satisfies
	\begin{align}
	\mathrm{Im} \int_{\mathbb T^n} \phi(x) dx < 0
	\quad \mathrm{or} \quad
	\mathrm{Re} \int_{\mathbb T^n} \phi(x) dx \neq 0,
	\label{eq:1.5}
	\end{align}
then there is no
$C^1((0,\infty);L^2(\mathbb T^n)) \cap C(\lbrack 0,\infty);(H^2 \cap L^{2p})(\mathbb T^n))$
function satisfying \eqref{eq:1.2} on $\lbrack 0, \infty)$ in the $L^2$ framework.
\end{Proposition}
Proposition \ref{Proposition:1.1} may be shown by an ODE argument.
Indeed, when \eqref{eq:1.5} holds,
there is a complex number $\alpha$ such that
$\mathrm{Re}(\alpha) > 0$ and
	\[
	M (t)
	:= - (2 \pi)^n \mathrm{Im} ( \alpha \hat u(t,0))
	= - \mathrm{Im} \bigg( \alpha \int_{\mathbb T^n} u(t,x) dx \bigg)
	\]
satisfies the estimate
    \[
    M (0)
    = - \mathrm{Im} \bigg( \alpha \int_{\mathbb T^n} \phi (x) dx \bigg)
    > 0
    \]
and the ordinary differential inequality
	\begin{align}
	\dot M(t) \geq C \mathrm{Re}(\alpha) M(t)^p
	\label{eq:1.6}
	\end{align}
with some $C > 0$.
Then \eqref{eq:1.6} implies that
$M$ is not bounded.
By a similar argument,
it is also seen that if $\phi$ satisfies
	\begin{align}
	\int_{\mathbb T^n} \phi(x) dx = 0
	\quad \mathrm{and} \quad \phi \neq 0,
	\label{eq:1.7}
	\end{align}
there is no global solution to \eqref{eq:1.2}.
For the sake of completion,
this approach is revisited in the last section.

For the non-existence of global solutions,
\eqref{eq:1.5} is the sharp condition
of $\hat \phi(0)$.
Indeed,
it is easy to see that if $\phi(x) \equiv i \mu_0$ with $\mu_0 \in \mathbb R$,
	\begin{align}
	u(t,x) =
	\begin{cases}
	i ( \frac{t}{p-1} + \mu_0^{-\frac{1}{p-1}} )^{-p+1}
	&\mathrm{if} \quad \mu_0 > 0,\\
	0
	&\mathrm{if} \quad \mu_0 = 0,\\
	- i (|\mu_0|^{-\frac{1}{p-1}}- \frac{t}{p-1})^{-p+1}
	&\mathrm{if} \quad \mu_0 < 0,\ t < (p-1)|\mu_0|^{-\frac{1}{p-1}}
	\end{cases}
	\label{eq:1.11}
	\end{align}
satisfies \eqref{eq:1.1} for $t > 0$ in a classical sense.
We note that
if $\mu_0 \geq 0$,
then $\phi = i \mu_0$ does not satisfy the condition \eqref{eq:1.5} and $u$ blows up at $t = (p-1)|\mu_0|^{-\frac{1}{p-1}}$.

Even though the sharp condition of $\hat \phi(0)$ is known,
as far as the authors know,
the necessary and sufficient condition of $\phi$
for the non-existence of global solutions was not known.
Especially,
in earlier works,
only the behavior of the Fourier $0$ mode of solutions has been focused
but the contribution of the oscillation of solutions has not been considered sufficiently.
It is because the Fourier $0$ mode of solutions
is simply controlled by the positivity of the nonlinearity.
We note that this is similar in other problems:
the Schr\"odinger equations on $\mathbb R^n$,
Damped wave equations $\mathbb R^n$, for example.
We refer the reader \cite{bib:8,IS19} and the references therein.
However, it is insufficient to focus only on the Fourier $0$ mode to study the sharp criteria of blowup phenomena.
For example, the initial data of \eqref{eq:1.7}
does not posses global solutions nor satisfy \eqref{eq:1.5}.
Therefore,
we need to see the contribution of the oscillation of solutions
in order to clarify the condition for global existence.

The aim of this paper
is to consider the sharp condition of initial data
for the non-existence of global solutions.
We say that $u$ is a $L^2$ global solution to \eqref{eq:1.1}
if $u \in C^1([0,\infty); H^{-2}) \cap C([0,\infty); L^2)$
satisfies \eqref{eq:1.1} for any $t \geq 0$ in the sense of $H^{-2}$.
Our main statement is the following:

\begin{Proposition}
\label{Proposition:1.2}
Let $\phi \in L^2$.
Then there exist $L^2$ global solutions
if and only if $\phi(x) = i \mu_0$ with $\mu_0 \geq 0$.
More precisely,
If $u$ is a $L^2$ global solution,
then $u$ is constant with respect to $x$
and satisfies $\mathrm{Im} \thinspace u(t,0) > 0$ for any $t \geq 0$.
\end{Proposition}

We remark that
$L^2$ local solutions to \eqref{eq:1.1}
may not be constructed by the standard contraction argument.
However, for example, Proposition \ref{Proposition:1.2} implies that
any $H^1$ solutions with non-constant initial data blow up at a finite time
because a standard contraction argument implies that
the blowup alternative holds in the $H^1$ framework.

We may show Proposition \ref{Proposition:1.2} by a contradiction argument.
We remark that when initial data does not satisfy \eqref{eq:1.5},
unlike earlier works,
it seems no longer possible to show the blowup phenomena
by ignoring the effect of oscillation.
In this paper,
we may focus on the nonlinear interaction between the Fourier $0$ mode and oscillation of solutions.
Namely, we show that, thanks to the oscillation of solutions $u$,
$\hat u(t,0) \leq 0$ and $u(t) \neq 0$ hold at some $t > 0$
nevertheless $\hat u(0,0)$ may be positive.
Then by taking $\hat u(t,0) \leq 0$ as initial data,
the argument of Proposition \ref{Proposition:1.1} implies 
that $u$ does not exist globally
in the $L^2$ framework.

In the next section,
we prepare the proof of Proposition \ref{Proposition:1.2}
by introducing a decomposition of solutions and a priori controls of $u$.
Proposition \ref{Proposition:1.1} is shown in the last section.

\section{Preparation}
In order to show Proposition \ref{Proposition:1.2},
as wee see in the next section,
it is sufficient to consider the case where
	\begin{align}
	\mathrm{Im} \int_{\mathbb T} \phi (x) dx > 0
	\quad \mathrm{and} \quad \mathrm{Re} \int_{\mathbb T} \phi (x) dx = 0.
	\label{eq:2.1}
	\end{align}
Under the condition \eqref{eq:2.1},
we remark that the identity
    \[
    \mathrm{Re} \int_{\mathbb T} u(t,x) dx = 0
    \]
holds a priori at any $t$.
We may rewrite our solution $u$ by
	\[
	u(t,x)
	= \sum_{k \in \mathbb Z} u_k(t) e^{i k x - i k^2 t}.
	\]
We remark that $u_k(t) = \hat u(t,k) e^{i k^2 t}$ and
$e^{i k x - i k^2 t}$ is a solution of linear Schr\"odinger equation
for any $k \in \mathbb Z$.
We note that the series above converges in $L^2$.
We also note that since $u$ is differentiable as a $H^{-2}$ valued function,
so is each Fourier coefficient of $u$.
Especially, $\dot u_k$ coincides with $(\partial_t u)_k$ for any $k$
and satisfies the estimate
	\[
	\sum_k |\dot{u_k}(t)|^2 (1+k^2)^{-1} < \infty,
	\]
where $\partial_t u$ is the time derivative of $u$ in the sense of $H^{-2}$.
The Fourier coefficient of the nonlinearity is computed by
	\begin{align}
   	\frac{1}{2 \pi} \int_{\mathbb T} |u(t,x)|^2 e^{-i k x} dx
	&= \sum_{\ell} \overline{u_{\ell}(t)} u_{k+\ell}(t) e^{- 2 i k \ell t}
	\label{eq:2.2}
	\end{align}
for any $k$.

Then \eqref{eq:1.1} is rewritten by the system,
	\begin{align}
	\dot \mu(t)
	&= - \mu(t)^2 - \sum_{\ell \ne 0} |u_\ell(t)|^2,
	\label{eq:2.3}\\
	\dot u_k(t)
	&= - i \sum_{\ell} \overline{u_\ell(t)} u_{k+\ell}(t) e^{- 2 i k \ell t}
	\nonumber\\
	&= - \mu(t) u_k(t) + \mu(t) \overline{u_{-k}(t)} e^{2 i k^2 t}
	- i \sum_{\ell \ne 0, -k} \overline{u_\ell(t)} u_{k+\ell}(t) e^{- 2 i k \ell t}
	\label{eq:2.4}
	\end{align}
for $k \ne 0$,
where
    \[
    \mu(t)=- i \ u_0(t)= \mathrm{Im} \ u_0(t).
    \]
We put
	\[
	\nu_K (t) = \sum_{0 < |k| \leq K} |u_k (t)|^2 \quad \mathrm{for} \quad K > 0,\quad
	\nu (t) = \sum_{k \neq 0} |u_k (t)|^2,
	\]
$\mu_0=\mu(0)$, and $\nu_0 = \nu(0)$.
By multiplying $\overline u_k$ by \eqref{eq:2.4} and summing up the resulting identity,
we get
	\begin{align}
	\dot \nu_K(t)
	= - 2 \mu(t) \nu_K(t) + 2 R_K(t),
	\label{eq:2.5}
	\end{align}
where
\begin{equation}\label{eq.dr1}
   	R_K(t)
	= \mathrm{Im} \bigg( \sum_{0 < |k| \leq K} \sum_{\ell \ne 0} \overline{u_k(t) u_\ell(t)} u_{\ell + k}(t) e^{-2ik \ell t} \bigg). 
\end{equation}

Now we show a uniform estimate of $R_K$.
\begin{Lemma} 
\label{Lemma:2.1}  There exist positive constants $A$ and $B$, such that for any $K>0$, $T>0$ and any sequence of functions
$\{u_k \}_{k \in \mathbb{Z}} $ satisfying

\begin{itemize} 
 \item $\mathrm{Re} \ u_0(t) =0, \ \mu(t)= \mathrm{Im} \ u_0(t) >0, $ at any $t \in \lbrack 0, T)$
     \item $\{u_k \}_{k \in \mathbb{Z}} \in C([0,T); \ell^2),$
     \item $u_k \in C^1((0,T))$ for any $k \in \mathbb{Z},$
     \item $\mu$ and $\{u_k \}_{k \neq 0}$ satisfy the equations in \eqref{eq:2.3} and \eqref{eq:2.4},
 \end{itemize}
 the estimate
	\begin{align}
	\left| \int_0^t R_K(\tau) d\tau \right|
	&\leq A( \mu(t)\nu(t) + \nu^{3/2}(t)
	+ \mu_0 \nu_0 + \nu_0^{3/2})
	\nonumber\\
	&+ B \int_0^t \left( \mu(\tau) \nu^{3/2}(\tau) + \nu^{2}(\tau)\right) d \tau ,
	\label{eq:2.6m}
	\end{align}
holds at any $t \in \lbrack 0,T)$. Here
$\mu_0$, $\nu$, and $\nu_0$ are defined as above
and $R_k$ is defined as in \eqref{eq.dr1}.
\end{Lemma}

\begin{proof}
We assume the RHS of \eqref{eq:2.6m} is finite.
By the Lebesgue dominant convergence theorem,
we have
	\[
	\int_0^t R_K(\tau ) d \tau
	= \mathrm{Im} \bigg( \sum_{0 < |k| \leq K} \sum_{\ell \ne 0} \int_0^t R_{k,\ell}(\tau) d \tau \bigg),
	\]
where we put
	\[
	R_{k,\ell}(t)
	= \overline{u_k(t) u_\ell(t)} u_{\ell + k}(t) e^{-2ik \ell t}.
	\]
When $\ell = -k$,
by integration by parts, \eqref{eq:2.3} and \eqref{eq:2.4} imply that we have
	\begin{align}
	&\int_0^t R_{k,-k}(\tau) d\tau
	\nonumber\\
	&= \frac{1}{2k^2}
	\bigg( \overline{u_k(t) u_{-k}(t)} \mu (t) e^{2 i k^2 t}
	- \overline{u_k(0) u_{-k}(0)} \mu (0) \bigg)
	\nonumber\\
	&+ \frac{1}{2 k^2} \int_0^t \overline{u_k(\tau) u_{-k}(\tau)} ( \mu^2(\tau) + \nu^2(\tau) )
	e^{2ik^2 \tau} d\tau
	\nonumber\\
	&+ \frac{1}{2 k^2} \int_0^t
	(2 \overline{u_k(\tau) u_{-k}(\tau)} e^{2ik^2 \tau} - |u_k(\tau)|^2 - |u_{-k}(\tau)|^2)
	\mu^2(\tau) d \tau
	\nonumber\\
	&+ \frac{i}{2 k^2} \int_0^t \overline{u_{-k}(\tau)} \mu_0(\tau) \sum_{m \ne 0,-k} u_m(\tau) \overline{u_{m+k}(\tau)} e^{2ik(m+k)\tau} d \tau
	\nonumber\\
	&+ \frac{i}{2 k^2} \int_0^t \overline{u_{k}(\tau)} \mu_0(\tau) \sum_{m \ne 0,k} u_m(\tau) \overline{u_{m-k}(\tau)} e^{2ik(k-m)\tau} d \tau.
	\label{eq:2.7}
	\end{align}
When $\ell \ne 0, -k$,
the identity
	\begin{align}
	&\int_0^t R_{k,\ell}(\tau) d \tau
	\nonumber\\
	&= \frac{i}{2k \ell}
	\bigg( \overline{u_k(t) u_{\ell}(t)} u_{k+\ell} (t) e^{- 2 i k \ell t}
	- \overline{u_k(0) u_{\ell}(0)} u_{k+\ell} (0) \bigg)
	\nonumber\\
	&+ \frac{i}{2 k \ell} (I_{k,\ell}+I_{\ell,k}+J_{k,\ell})
	\label{eq:2.8}
	\end{align}
holds similarly, where
	\begin{align}
	I_{k,\ell}
	&= - \int_0^t \overline{\dot{u}_k(\tau) u_{\ell}(\tau)} u_{k+\ell}(\tau) e^{-2ik\ell \tau} d \tau
	\nonumber\\
	&= \int_0^t \mu(\tau) \overline{u_k(\tau) u_{\ell}(\tau)} u_{k+\ell}(\tau) e^{-2ik\ell \tau}
	- \mu(\tau) u_{-k}(\tau) \overline{u_{\ell}(\tau)} u_{k+\ell}(\tau) e^{-2ik(\ell + k) \tau} d \tau
	\nonumber\\
	&+ i \sum_{m \ne 0, -k} \int_0^t u_m(\tau) \overline{u_{m+k}(\tau) u_{\ell}(\tau)} u_{k+\ell}(\tau) e^{-2ik(\ell-m) \tau} d \tau
	\label{eq:2.9}
	\end{align}
and
	\begin{align}
	J_{k,\ell}
	&= - \int_0^t \overline{u_\ell(\tau) u_{k}(\tau)} \dot u_{k+\ell}(\tau) e^{-2ik\ell \tau} d \tau
	\nonumber\\
	&= \int_0^t \mu(\tau) \overline{u_k(\tau) u_{\ell}(\tau)} u_{k+\ell}(\tau) e^{-2ik\ell \tau} d \tau
	\nonumber\\
	&- \int_0^t \overline{u_{-k-\ell}(\tau) u_{k}(\tau) u_{\ell}(\tau)} \mu(\tau) e^{-2i(k \ell - (k +\ell)^2) \tau}
	\nonumber\\
	&+ i \sum_{m \ne 0, -k -\ell} \int_0^t \overline{u_m (\tau) u_{k} (\tau) u_{\ell} (\tau)} u_{m+k+\ell}(\tau) e^{-2i(k \ell + k m + m \ell) \tau} d \tau.
	\label{eq:2.10}
	\end{align}
We recall that for positive sequences $a$, $b$, $c$, $d$,
the estimates
	\begin{align*}
	\sum_{k \ne 0} \sum_{\ell \ne 0,-k} \frac{1}{k \ell} a_k b_\ell c_{k+\ell}
	&\leq \| k^{-1} a_k \|_{\ell_{k\ne 0}^1} \| \sum_{\ell \ne 0, -k} b_\ell c_{k+\ell} \|_{\ell_{k\neq0}^\infty}\\
	&\leq \frac{\pi}{\sqrt{3}} \| a_k \|_{\ell_{k\neq0}^2} \| b_k \|_{\ell_{k\neq0}^2} \| c_k \|_{\ell_{k\neq0}^2},
	\end{align*}
	\begin{align*}
	&\sum_{k \ne 0} \sum_{\ell \ne 0,-k} \sum_{m \ne 0,-k} \frac{1}{k \ell} a_m b_{m+k} c_{\ell} d_{\ell+k}\\
	&\leq \| \sum_{m \ne 0,-k} a_m b_{m+k} \|_{\ell_{k \ne 0}^\infty}
	\| \ell^{-1} c_\ell \|_{\ell_{\ell \ne 0}^1}
	\| k^{-1} d_{\ell+k} \|_{\ell_{\ell \ne 0}^\infty(\ell_{k \ne -\ell}^1)}\\
	&\leq \frac{\pi^2}{3} \| a_k \|_{\ell_{k\neq0}^2} \| b_k \|_{\ell_{k\neq0}^2}
	\| c_k \|_{\ell_{k\neq0}^2} \| d_k \|_{\ell_{k\neq0}^2},
	\end{align*}
and
	\begin{align*}
	&\sum_{k \ne 0} \sum_{\ell \ne 0} \sum_{m \ne 0,-k-\ell} \frac{1}{k \ell} a_k b_{\ell} c_{m} d_{m+k+\ell}\\
	&\leq \| k^{-1} a_k \|_{\ell_{k \ne 0}^1} \| k^{-1} b_k \|_{\ell_{k\neq0}^1}
	\| c_k \|_{\ell_{k\neq0}^2} \| d_k \|_{\ell_{k\neq0}^2}\\
	&\leq \frac{\pi^2}{3} \| a_k \|_{\ell_{k\neq0}^2} \| b_k \|_{\ell_{k\neq0}^2}
	\| c_k \|_{\ell_{k\neq0}^2} \| d_k \|_{\ell_{k\neq0}^2}
	\end{align*}
hold, where we have used the fact that
	\[
	\sum_{k=1}^\infty k^{-2} = \frac{\pi^2}{6}.
	\]
These estimates above and \eqref{eq:2.7}, \eqref{eq:2.8},
\eqref{eq:2.9}, and \eqref{eq:2.10}
imply \eqref{eq:2.6m}.
\end{proof}

Thanks to Lemma \ref{Lemma:2.1},
we have the following a priori controls:
\begin{Lemma}
\label{Lemma:2.2}
Under the assumption of Lemma \ref{Lemma:2.1},
if $\mu_0$ and $\nu_0$ are small enough to satisfy
the estimate
	\[
	4 A \mu_0 + 4 A (2 \nu_0)^{1/2} + 3 B \mu_0
	< \frac 1 2
	\]
with $A$ and $B$ of Lemma \ref{Lemma:2.1}, then the following estimate holds:
	\begin{align}
	\sup_{t \in  \lbrack 0, T)} \nu(t) \leq 2 \nu_0.
	\label{eq:2.11}
	\end{align}
\end{Lemma}

\begin{proof}
The positivity of $\mu$ and $\nu$, and \eqref{eq:2.5} imply that
$\nu$ satisfies
	\begin{align}
	\nu(t)
	&= \lim_{K \to \infty} \nu_K(t)
	\nonumber\\
	&\leq \nu_0 + \limsup_{K \to \infty} \bigg( 2 \int_0^t - \mu(\tau) \nu_K(\tau) + R_K(\tau) d\tau \bigg)
	\nonumber\\
	&\leq \nu_0 + 2 \limsup_{K \to \infty} \bigg| \int_0^t R_K(\tau) d\tau \bigg|.
	\label{eq:2.12}
	\end{align}
Let $M(t) = \sup_{\tau \in \lbrack 0,t \rbrack} \nu(\tau)$.
Since $\mu$ is non-negative and decreasing on $[0,T)$,
Lemma \ref{Lemma:2.1} and \eqref{eq:2.12} imply that
the estimate
	\begin{align}
	M(t)
	&\leq \nu_0 + 4 A ( \mu_0 M(t) + M^{3/2}(t))
	+ 2 B \int_0^t \mu(\tau) \nu(\tau)^{3/2} + \nu(\tau)^2 d \tau
	\label{eq:2.13}
	\end{align}
holds.
Thanks to the monotonicity of $\mu$,
we can change variable of the integral in the right hand side of \eqref{eq:2.13}
by
	\[
	t \in \lbrack 0,T) \to s = \mu_0 - \mu(t) \in \lbrack 0, \mu_0 - \mu(T) ).
	\]
We put $V(s) = \nu(t)$ and
	\[
	T_1 = \sup\{ t \in (0 ,T ) \ \mid \ M(t) \leq 2 \nu_0 \}.
	\]
Then for $t \in [0,T_1]$, we have
	\begin{align*}
	M(t)
	&\leq \nu_0 + 4 A ( \mu_0 M(t) + M(t)^{3/2})
	+ 2 B \int_0^{\mu_0-\mu(t)}
	\frac{(\mu_0 - \sigma) V(\sigma)^{3/2} + V(\sigma)^2}{(\mu_0-\sigma)^2 + V(\sigma)} d \sigma\\
	&\leq \nu_0 + (4 A \mu_0 + 4 A M(t)^{1/2} + 3 B \mu_0) M(t)\\
	&< \nu_0 + \frac 1 2 M(t),
	\end{align*}
where we have used the Young inequality,
	\begin{align}
	2 (\mu_0 - \sigma) V(\sigma)^{1/2} \leq (\mu_0-\sigma)^2 + V(\sigma).
	\label{eq:2.15}
	\end{align}
Therefore $T_1=T$ and \eqref{eq:2.11} is shown.
\end{proof}

The next statement is a direct consequence of Lemma \ref{Lemma:2.2}.
\begin{Corollary}
\label{Corollary:2.3}
Under the assumption of Lemma \ref{Lemma:2.1},
if $\mu_0$ and $\nu_0$ are sufficiently small,
then the estimate
	\[
	\delta
	:=
	\inf_{t \in \lbrack 0 , T)}
	\big[ \nu_0
	- 2 A \big( \mu(t) \nu(t) + \nu(t)^{3/2} + \mu_0 \nu_0 + \nu_0^{3/2} \big) \big]
	> 0
	\]
holds with $A$ of Lemma \ref{Lemma:2.1}.
\end{Corollary}

\section{Proof of Proposition \ref{Proposition:1.2}}
We, at first, omit the case
where $u$ is given by \eqref{eq:1.11},
namely the case where $u$ is constant with respect to space.

When $\mathrm{Im} (\hat \phi(0)) \leq 0$,
by integrating \eqref{eq:1.4} and taking the imagniary part of the resulting equation,
$\mu(t) = \mathrm{Im} (\hat u(t,0))$ satisfies a priori the estimate
    \[
    \dot \mu(t) < - \mu(t)^2.
    \]
Therefore $\mu(t)$ goes to $-\infty$ at a finite time.
When $\mathrm{Re} (\hat \phi(0)) \neq 0$,
by multiplying some $\alpha \in \mathbb C$
with $\mathrm{Re} (\alpha) >0$ and $\mathrm{Im}( \alpha \hat \phi(0)) < 0$ by \eqref{eq:1.4}
and repeating a similar argument with $\alpha u$,
$\mathrm{Im} (\alpha \hat u(t,0))$ is shown to go to $-\infty$ at a finite time.

Now we consider the case of \eqref{eq:2.1},
where $\mathrm{Im} (\hat \phi(0)) > 0$ and
$\mathrm{Re} (\hat \phi(0)) = 0$.
Namely, we consider the case where \eqref{eq:1.4}
is rewritten by the system of \eqref{eq:2.3} and \eqref{eq:2.4}
with $\mu_0 > 0$ and $\nu_0 > 0$.
We assume $u$ exists at any positive time and show the contradiction.
In order to show the contradiction,
it is sufficient to show that there exists $T_0 \in (0, \infty)$
such that $\mu(T_0)=0$ and $\nu(T_0) > 0$ holds.
Indeed, if such $T_0$ exists,
the argument above implies that $\mu(t)$ goes $-\infty$
at a finite time.
We note that the equation
	\[
	\dot \mu = - \mu^2 - \nu^2
	\]
implies that there exists $T_0 \in (0,\infty]$ such that $\mu(T_0)=0$
and $\mu$ does not converge as long as $\nu$ is positive.
We also note that there does not exist $T_1 \in (0, \infty)$
such that $\mu(T_1)=\nu(T_1)=0$.
If such $T_1$ exists, then $u(t)=0$ holds for $t \geq 0$.

Without loss of generality,
we assume that $\mu_0$ and $\nu_0$ are so small
that Corollary \ref{Corollary:2.3} is applicable.
\eqref{eq:2.5} and Lemma \ref{Lemma:2.1}, Corollary \ref{Corollary:2.3} imply that $\nu$ enjoys the estimate,
	\[
	\nu(t)
	\geq \delta - 2 \int_{0}^t \mu(\tau) \nu(\tau) + B \mu(\tau) \nu(\tau)^{3/2} + B \nu(\tau)^2 d\tau
	\]
with some $\delta > 0$ for any $t \in [0,T_0]$.
Then we change $t$ into $s = \mu_0 - \mu(t)$
and put $\nu(t) = V(s)$ as the proof of Lemma \ref{Lemma:2.2}.
Then the estimate above implies that
we have
	\begin{align}\label{eq:3.1}
	V(s)
	&\geq \delta
	- 2 \int_{0}^s \frac{(\mu_0-\sigma) V(\sigma)
	+ B (\mu_0-\sigma) V(\sigma)^{3/2}
	+ B V(\sigma)^2}{(\mu_0-\sigma)^2+V(\sigma)} d\sigma,
	\nonumber\\
	&\geq \delta
	- \int_{0}^s 2 \frac{(\mu_0-\sigma) V(\sigma)}{(\mu_0-\sigma)^2+V(\sigma)}
	+ 3 B V(\sigma) d\sigma,
	\end{align}
where we have used the Young inequality \eqref{eq:2.15}.
Now we claim that there exists $g \in C([0,\mu_0]) \cap C^1((0,\mu_0))$
satisfying the estimates
	\[
	\dot g(s) \leq - 2 \frac{(\mu_0-s) g(s)}{(\mu_0-s)^2+g(s)} - 3 B g(s),
	\]
$g(0) \leq \delta $, and $g(\mu_0) > 0$.
Then by the comparison argument,
we have
	\[
	\nu(T_0) = V(\mu_0) \geq g(\mu_0) > 0
	\]
and this shows the contradiction.

In order to find $g$ above,
for $s \in \lbrack 0, \mu_0)$, we put
	\[
	f(s) = \frac{(\mu_0-s)^2}{W(C_1(\mu_0-s)^2)},
	\quad
	C_1 = \frac{e^{\mu_0^2/f_0}}{f_0}
	\]
with $f_0 > 0$,
where $W$ is the Lambert $W$ function given by
	\[
	W(\sigma) e^{W(\sigma)} = \sigma
	\]
for $\sigma \geq 0$.
For the detail of the Lambert $W$ function,
we refer \cite{CGHJK96}, for example.
We note that the identity $f(0) = f_0$ holds
because the definition implies that
	\[
	W(\sigma e^\sigma) = \sigma
	\]
holds for $\sigma \geq 0$.
We also note that the identity
	\[
	\lim_{s \nearrow \mu_0} f(s) = f_0 e^{- \mu_0^2/f_0} > 0
	\]
follows from \cite[(3.1)]{CGHJK96}.
Now we claim that $f$ satisfies the identity
	\begin{align}
	\dot f(s)
	= - \frac{2 (\mu_0-s) f(s)}{(\mu_0-s)^2 + f(s)}
	\label{eq:3.2}
	\end{align}
for $s \in (0, \mu_0)$.
Indeed,
since $W$ enjoys the identity
	\[
	\dot W(\sigma) = \frac{W(\sigma)}{\sigma(1+W(\sigma))}
	\]
for $\sigma > 0$ (See \cite[(3.2)]{CGHJK96}),
we compute
	\begin{align*}
	\dot f(s)
	&= - 2 \frac{\mu_0-s}{W(C_1(\mu_0-s)^2)}
	+ 2 C_1 \frac{(\mu_0-s)^3}{W(C_1(\mu_0-s)^2)^2} \dot W(C_1(\mu_0-s)^2)\\
	&= - 2 \frac{\mu_0-s}{W(C_1(\mu_0-s)^2)}
	+ 2 \frac{\mu_0-s}{W(C_1(\mu_0-s)^2)(1+W(C_1(\mu_0-s)^2))}\\
	&= - 2 \frac{\mu_0-s}{1+W(C_1(\mu_0-s)^2)}\\
	&= - 2 \frac{(\mu_0-s)f(s)}{(\mu_0-s)^2+f(s)}.
	\end{align*}
Then we put $g(s) = e^{3 B (\mu_0-s)} f(s)$.
Since $g(s) \geq f(s)$ for $s \in \lbrack 0,\mu_0)$,
$\dot g$ is estimated by
	\begin{align*}
	\dot g (s)
	&= - \frac{2 (\mu_0-s) g(s)}{(\mu_0-s)^2+f(s)} - 3 B g(s)\\
	&\leq - \frac{2 (\mu_0-s) g(s)}{(\mu_0-s)^2+g(s)} - 3 B g(s).
	\end{align*}
Therefore, if $f_0 < e^{-3 B \mu_0} \delta$,
then the desired $g$ is obtained.


\section*{Acknowledgment}
The first author is supported
by JSPS Grants-in-Aid for JSPS Fellows 19J00334
and JSPS Early-Career Scientists 20K14337. The second  author was supported in part by  INDAM, GNAMPA - Gruppo Nazionale per l'Analisi Matematica, la Probabilit\`a e le loro Applicazioni, by Institute of Mathematics and Informatics, Bulgarian Academy of Sciences, by Top Global University Project, Waseda University and  the Project PRA 2018 49 of  University of Pisa.


\end{document}